\documentclass[11pt,letterpaper]{amsart}
\usepackage[usenames,dvipsnames]{color}
\usepackage{amsthm,amsfonts,amssymb,amsmath,amsxtra}
\usepackage[all]{xy}
\SelectTips{cm}{}
\usepackage{xr-hyper}
\usepackage[colorlinks=true,
   citecolor=Black,
   linkcolor=Red,
   urlcolor=Blue]{hyperref}
\usepackage{verbatim}
\usepackage{caption}

\usepackage{lineno}

\usepackage{mathrsfs}

\RequirePackage{xspace}
\RequirePackage{etoolbox}
\RequirePackage{varwidth}
\RequirePackage{enumitem}
\RequirePackage{tensor}
\RequirePackage{mathtools}
\RequirePackage{longtable}
\RequirePackage{multirow}

\def\ge{\geqslant}
\def\le{\leqslant}
\def\a{\alpha}

\def\g{\gamma}

\def\L{\Lambda}
\def\e{\epsilon}

\def\o{\omega}

\def\th{\theta}
\def\k{\kappa}

\def\i{^{-1}}

\def\<{\langle}
\def\>{\rangle}

\newcommand{\BE}{\ensuremath{\mathbb {E}}\xspace}
\newcommand{\BF}{\ensuremath{\mathbb {F}}\xspace}
\newcommand{{\BG}}{\ensuremath{\mathbb {G}}\xspace}
\newcommand{\BH}{\ensuremath{\mathbb {H}}\xspace}
\newcommand{\BI}{\ensuremath{\mathbb {I}}\xspace}

\newcommand{{\BK}}{\ensuremath{\mathbb {K}}\xspace}
\newcommand{\BL}{\ensuremath{\mathbb {L}}\xspace}

\newcommand{\BQ}{\ensuremath{\mathbb {Q}}\xspace}
\newcommand{\BR}{\ensuremath{\mathbb {R}}\xspace}
\newcommand{\BS}{\ensuremath{\mathbb {S}}\xspace}
\newcommand{\BT}{\ensuremath{\mathbb {T}}\xspace}
\newcommand{\BU}{\ensuremath{\mathbb {U}}\xspace}
\newcommand{\BV}{\ensuremath{\mathbb {V}}\xspace}
\newcommand{\BW}{\ensuremath{\mathbb {W}}\xspace}
\newcommand{\BX}{\ensuremath{\mathbb {X}}\xspace}

\newcommand{\BZ}{\ensuremath{\mathbb {Z}}\xspace}

\newcommand{\CA}{\ensuremath{\mathcal {A}}\xspace}
\newcommand{\CB}{\ensuremath{\mathcal {B}}\xspace}
\newcommand{\CC}{\ensuremath{\mathcal {C}}\xspace}

\newcommand{\CG}{\ensuremath{\mathcal {G}}\xspace}
\newcommand{\CH}{\ensuremath{\mathcal {H}}\xspace}

\newcommand{\CK}{\ensuremath{\mathcal {K}}\xspace}

\newcommand{\CO}{\ensuremath{\mathcal {O}}\xspace}
\newcommand{\CP}{\ensuremath{\mathcal {P}}\xspace}

\newcommand{\CS}{\ensuremath{\mathcal {S}}\xspace}

\newcommand{\CX}{\ensuremath{\mathcal {X}}\xspace}

\newcommand{\GL}{\mathrm{GL}}

\DeclareMathOperator{\tr}{tr}

\newcommand{\ov}{\overline}

\def\brk{{\breve k}}

\def\COk{{\CO_{\brk}}}

\def\tPhi{\widetilde \Phi}

\def\ind{{\rm ind}}

\def\bx{{\mathbf x}}

\def\der{{\rm der}}

\def\ov{\overline}

\newtheorem{theorem}{Theorem}
\newtheorem{proposition}[theorem]{Proposition}
\newtheorem{lemma}[theorem]{Lemma}

\newtheorem{corollary}[theorem]{Corollary}

\theoremstyle{definition}

\newtheorem{remark}[theorem]{Remark}

\numberwithin{equation}{section}
\numberwithin{theorem}{section}

\setitemize[0]{leftmargin=*,itemsep=\the\smallskipamount}
\setenumerate[0]{leftmargin=*,itemsep=\the\smallskipamount}

\renewcommand{\to}{%
   \ifbool{@display}{\longrightarrow}{\rightarrow}%
   }
\let\shortmapsto\mapsto
\renewcommand{\mapsto}{%
   \ifbool{@display}{\longmapsto}{\shortmapsto}%
   }
\newlength{\olen}
\newlength{\ulen}
\newlength{\xlen}
\newcommand{\xra}[2][]{%
   \ifbool{@display}%
      {\settowidth{\olen}{$\overset{#2}{\longrightarrow}$}%
       \settowidth{\ulen}{$\underset{#1}{\longrightarrow}$}%
       \settowidth{\xlen}{$\xrightarrow[#1]{#2}$}%
       \ifdimgreater{\olen}{\xlen}%
          {\underset{#1}{\overset{#2}{\longrightarrow}}}%
          {\ifdimgreater{\ulen}{\xlen}%
             {\underset{#1}{\overset{#2}{\longrightarrow}}}
             {\xrightarrow[#1]{#2}}}}%
      {\xrightarrow[#1]{#2}}
   }
\makeatother
\newcommand{\xyra}[2][]{%
   \settowidth{\xlen}{$\xrightarrow[#1]{#2}$}%
   \ifbool{@display}%
      {\settowidth{\olen}{$\overset{#2}{\longrightarrow}$}%
       \settowidth{\ulen}{$\underset{#1}{\longrightarrow}$}%
       \ifdimgreater{\olen}{\xlen}%
          {\mathrel{\xymatrix@M=.12ex@C=3.2ex{\ar[r]^-{#2}_-{#1} &}}}%
          {\ifdimgreater{\ulen}{\xlen}%
             {\mathrel{\xymatrix@M=.12ex@C=3.2ex{\ar[r]^-{#2}_-{#1} &}}}
             {\mathrel{\xymatrix@M=.12ex@C=\the\xlen{\ar[r]^-{#2}_-{#1} &}}}}}%
      {\mathrel{\xymatrix@M=.12ex@C=\the\xlen{\ar[r]^-{#2}_-{#1} &}}}%
   }
\makeatletter
\newcommand{\xla}[2][]{%
   \ifbool{@display}%
      {\settowidth{\olen}{$\overset{#2}{\longleftarrow}$}%
       \settowidth{\ulen}{$\underset{#1}{\longleftarrow}$}%
       \settowidth{\xlen}{$\xleftarrow[#1]{#2}$}%
       \ifdimgreater{\olen}{\xlen}%
          {\underset{#1}{\overset{#2}{\longleftarrow}}}%
          {\ifdimgreater{\ulen}{\xlen}%
             {\underset{#1}{\overset{#2}{\longleftarrow}}}
             {\xleftarrow[#1]{#2}}}}%
      {\xleftarrow[#1]{#2}}
   }
\newcommand{\isoarrow}{%
   \ifbool{@display}{\overset{\sim}{\longrightarrow}}{\xrightarrow\sim}%
   }

\newcommand{\sm}{{\,\smallsetminus\,}}

\usepackage{upgreek}
\makeatletter
\newcommand{\colim@}[2]{%
  \vtop{\m@th\ialign{##\cr
    \hfil$#1\operator@font lim$\hfil\cr
    \noalign{\nointerlineskip\kern1.5\ex@}#2\cr
    \noalign{\nointerlineskip\kern-\ex@}\cr}}%
}
\newcommand{\colim}{%
  \mathop{\mathpalette\colim@{\rightarrowfill@\textstyle}}\nmlimits@
}
\newcommand{\prolim@}[2]{%
  \vtop{\m@th\ialign{##\cr
    \hfil$#1\operator@font lim$\hfil\cr
    \noalign{\nointerlineskip\kern1.5\ex@}#2\cr
    \noalign{\nointerlineskip\kern-\ex@}\cr}}%
}
\newcommand{\prolim}{%
  \mathop{\mathpalette\colim@{\leftarrowfill@\textstyle}}\nmlimits@
}
\makeatother

\begin{document}
\title[]{An explicit decomposition of higher Deligne-Lsuztig representations}

\author[Ben Liu]{Ben Liu}
\address{Academy of Mathematics and Systems Science, Chinese Academy of Sciences, Beijing 100190, China}
\email{liubenmath@gmail.com}

\author[Sian Nie]{Sian Nie}
\address{Academy of Mathematics and Systems Science, Chinese Academy of Sciences, Beijing 100190, China}

\address{School of Mathematical Sciences, University of Chinese Academy of Sciences, Chinese Academy of Sciences, Beijing 100049, China}
\email{niesian@amss.ac.cn}

\begin{abstract}
In a previous paper \cite{Nie_24}, the second named author obtains a decomposition of an elliptic higher Deligne-Lusztig representation into irreducible summands, which are built in the same way as Yu types using a geometric analog $\k'$ of the Weil-Heisenberg representation $\k$. In this note, we show that $\k'$ and $\k$ differs by a character $\chi$.  Moreover, under a mild condition on the cardinality $q$ of the residue field (for instance $q > 3$), we show that $\chi$ equals the sign character constructed by Fintzen-Kaletha-Spice \cite{FintzenKS}, which gives an explicit irreducible decomposition result on elliptic higher Deligne-Lusztig representations. As an application, we deduce (under the mild condition on $q$) that each unramified Yu type appears in the cohomology of higher Deligne-Lusztig varieties, and each unramified Kaletha's regular supercuspidal representation is the compact induction of a specified higher Deligne-Lusztig representation up to a sign.
\end{abstract}

\maketitle

\section{Introduction}
In \cite{Lusztig_79}, Lusztig introduced higher Deligne-Lusztig representations for parahoric subgroups, which are natural extensions of the celebrated Deligne-Lusztig representations for finite groups of Lie type \cite{DeligneL_76}. Motivated by the work of G\'{e}rardin \cite{Ger}, Lusztig conjectured that compact inductions of higher Deligne-Lusztig representations will realize supercuspidal representations of $p$-adic groups. Recently, there have been intensive studies on Lusztig's conjecture and further applications to (modular) local Langlands program. Let us mention the work  \cite{BoyarchenkoW_16}, \cite{Chan_siDL}, \cite{CI_loopGLn}, \cite{ChenS_17}, \cite{ChenS_23}, \cite{ChanOi_25a}, \cite{ChanOi_25b}, \cite{Nie_24}, \cite{IvanovNie_24}, \cite{IvanovNie_25} and references therein.

\subsection{Main result} The most interesting higher Deligne-Lusztig representations are those attached to attached to elliptic tori, which are referred to as elliptic higher Deligne-Lusztig representations. All the other higher Deligne-Lusztig representations are parabolic inductions of elliptic ones. The main purpose of note is to give an explicit algebraic characterization of elliptic higher Deligne-Lusztig representation.

Let $p \neq \ell$ be two different prime numbers. Let $k$ be a non-archimedean local field with a finite residue field of cardinality $q$ and of characteristic $p$. Let $G$ be a connected $k$-rational reductive group which splits over a maximal unramifed extension $\brk$ of $k$. Let $Z_G$ denote the center of $G$. Fix a point $\bx$ in the Bruhat-Tits building $\CB(G, k)$ of $G$ over $k$. Denote by $\CG = \CG_\bx$ the attached (connected) parahoric subgroup over the integer ring $\CO_k$ of $k$. We recall the following condition on $p$:

\[\tag{*} \text{ $p \neq 2$ is not a bad prime for $G$ and $p \nmid |\pi_1(G_\der)| \cdot |\pi_1(\widehat G_\der)|$ }.\] Here $G_\der$ and $\widehat G_\der$ are the derived subgroups of $G$ and its dual group $\widehat G$ respectively.

Let $T \subseteq G$ be a $k$-rational and $\brk$-split maximal elliptic torus, whose apartment over $k$ contains $\bx$. $\phi: T(k) \to \ov\BQ_\ell^\times$ of depth $\le r \in \BZ_{\ge 0}$. Choose a $\brk$-rational maximal unipotent subgroup $U \subseteq G$ normalized by $T$.  Following \cite{Lusztig_79} and \cite{CI_MPDL}, one can associate an $\ov\BF_q$-variety \[X = X_{T, U, x, r},\] called a higher Deligne-Lusztig variety. Applying Deligne-Lusztig cohomological induction to $X$, one obtains a virtual smooth $Z_G(k)\CG(\CO_k)$-module \[R_{T, r}^G(\phi) = R_{T, U, \bx, r}^G(\phi),\] which is referred to as a higher Deligne-Lusztig representation. If $r = 0$, $R_{T, r}^G(\phi)$ is by definition a classical Deligne-Lusztig representation.

Suppose that $\phi$ has a Howe factorization $(\L, \phi_{-1})$ in the sense of \cite{Kaletha_19}, where $\L =(G^i, \phi_i, r_i)_{0 \le i \le d}$ is a generic datum and $\phi_{-1}: T(k) \to \ov\BQ_\ell^\times$ a depth $0$ character. Under the condition (*), the second named author \cite{Nie_24} proved the following result  \[R_{T, r}^G(\phi) = \ind_{\CK(\CO_k)}^{\CG(\CO_k)} \k_\L \otimes R_{T, 0}^{G^0}(\phi_{-1}),\] where $\CK$ is the identity component of a Yu-type $\CO_k$-subgroup associated to $(\L, \bx)$; $\k_\L$ is an irreducible $\CK(\CO_k)$-module obtained through a cohomological method; and $R_{T, 0}^{G^0}(\phi_{-1})$ is the classical Deligne-Lusztig representation.

Thus, to give a complete description of $R_{T, r}^G(\phi)$, it remains to determine the representation $\k_\L$ arising from geometry. To this end, we recall two classical representations of Yu's subgroup $\widetilde K \supseteq Z_G(k)\CK(\CO_k)$ associated to $(\L, \bx)$. The first one is the Weil-Heisenberg representation $\widetilde\k(\L)$ introduced by Yu \cite{Yu_01}, which play an essential role in the construction of supercuspidal representations for $p$-adic groups. The other is the (inflated) quadratic character $\widetilde\e_\L$ constructed by Fintzen, Kaletha and Spice \cite{FintzenKS}, which paly a delicate role in constructing supercuspidal $L$-packets (see also the work by Chan and Oi \cite{ChanOi_25a} when $\phi$ is toral). Both $\widetilde\k(\L)$ and $\widetilde\e_\L$ are constructed in a purely algebraic way and have very explicit descriptions. Put $\k(\L) = \widetilde\k(\L) |_{Z_G(k)\CK(\CO_k)}$ and $\e_\L = \widetilde\e_\L |_{Z_G(k)\CK(\CO_k)}$ for simplicity.

Now we state the main result of this note.
\begin{theorem} \label{main}
    Assume $p \neq 2$. Then there exists a character $\psi$ of $Z_G(k)\CK(\CO_k)$ such that \[(-1)^{d_\L} \k_\L = \k(\L) \otimes \e_\L \otimes \psi,\] where $d_\L \in \{\pm 1\}$ is a sign associated to $\L$ (see Theorem \ref{geo-rep}).

    If, moreover, $q \ge c_\L$ then $\psi = 1$. Here $2 \le c_\L \le 4$ is an integer defined in Proposition \ref{der}.
\end{theorem}
\begin{remark}
    It would be interesting to extend Theorem \ref{main} to the modular case. This would give an explicit description of modular higher Deligne-Lusztig representations, see \cite{IvanovNie_25}.
\end{remark}

Combining \cite[Theorem 1.6]{Nie_24} and Theorem \ref{main} we have an explicit irreducible decomposition result.
\begin{corollary} \label{algebraisation}
    Assume $p$ satisfies (*) and $q \ge c_\L$. Then \[R_{T, r}^G(\phi) = (-1)^{d_\L} \sum_{\rho} m_\rho \ind_{\CK(\CO_k)}^{\CG(\CO_k)} \k(\L) \otimes \e_\L \otimes \rho,\] where $\rho$ runs through irreducible factors up to isomorphism of $R_{T, 0}^{G^0}(\phi_{-1})$ with multiplicity $m_\rho \in \BZ$.  Moreover, the $\CG(\CO_k)$-modules \[\ind_{\CK(\CO_k)}^{\CG(\CO_k)} \k(\L) \otimes \e_\L \otimes \rho\] on the right hand side of (*) are irreducible and non-isomorphic with each other.
\end{corollary}

\begin{remark} \label{compare}
    When a preliminary version of this note is finished, we noticed that Chan and Oi \cite{ChanOi_25b} also obtains Corollary \ref{algebraisation} under certain largeness condition on $q$. Their approach is quite different from ours, which is based on studies of the regular case and  Green functions of higher Deligne-Lusztig representations. Note that our assumption on $q$ may be weaker than theirs in some cases. For instance, if $\phi$ is toral, then $c_\L = 2$ by definition and hence there is no restriction on $q$ in Corollary \ref{algebraisation}.
\end{remark}

\subsection{Application} Now we discuss some applications of the main result. In \cite{Yu_01}, Yu introduced the notion of (tame) cuspidal $G$-data, and associated with each cuspidal $G$-datum $\Sigma$ an irreducible supercuspidal representation $\pi_\Sigma^{\rm Yu}$ of $G(k)$. Conversely, thanks to work of Kim \cite{Kim} and Fintzen \cite{Fintzen_21_ann}, all the irreducible supercuspidal representations are exhausted by Yu's representations $\pi_\Sigma^{\rm Yu}$, provided $p$ does not divide the order of the absolute Weyl group of $G$.

Recall briefly that a cupidal $G$-datum $\Sigma$ consists of a generic datum $\L =(G^i, \phi_i, r_i)_{0 \le i \le d}$, a point $\bx \in \CB(G^0, k)$ and a depth $0$ irreducible representation $\widetilde \rho$ of $G^0(k)_{[\bx]}$ satisfying certain conditions. Here $G^0(k)_{[\bx]}$ denotes the stabilizer in $G^0(k)$ of the image $[\bx]$ of $\bx$ in $\CB(G^0_\der, k)$. Then $\pi_\Sigma^{\rm Yu}$ and its twisted company $\pi_\Sigma^{\rm FKS}$ are given by \[\pi_\Sigma^{\rm Yu} = \text{c-}\ind_{\widetilde K}^{G(k)} \widetilde \k(\L) \otimes \widetilde \rho, \quad \pi_\Sigma^{\rm FKS} = \text{c-}\ind_{\widetilde K}^{G(k)} \widetilde \k(\L) \otimes \widetilde\e_\L \otimes \widetilde \rho.\]

Now we assume the cuspidal datum $\Sigma = (\L, \bx, \widetilde\rho)$ above is unramified, that is, the Levi subgroup $G^0$ splits over $\brk$. Let $\CS_{\L, x}$ be the set of $k$-rational and $\brk$-split maximal tori of $G^0$ whose apartments over $k$ contain $\bx$. The following result gives an explicit relation between irreducible supercupidal representations and cohomology of convex higher Deligne-Lusztig varieties.
\begin{theorem} \label{main-sup}
    Assume $p$ satisfies (*) and $q \ge c_\L$. Let $\Sigma = (\L, \bx, \widetilde\rho)$ be an unramified cuspidal $G$-datum. Let $T \in \CS_{\L, \bx}$ be an elliptic torus and $\phi_{-1}$ a depth $0$ character of $T(k)$ such that $\hom_{Z_G(k)\CG^0(\CO_k)}(\widetilde\rho, R_{T, 0}^{G^0}(\phi_{-1})) \neq 0$. Then there exists $i \in \BZ_{\ge 0}$ such that \[\pi_\Sigma^{\rm FKS} \text{ is a direct summand of } \text{c-}\ind_{Z_G(k)\CG(\CO_k)}^{G(k)} H_c^i(X, \ov\BQ_\ell)[\phi],\] where $\phi$ is the unique character with Howe factorization $(\L, \phi_{-1})$.

    Moreover, such $T$ and $\phi_{-1}$ satisfying above conditions always exist.
\end{theorem}

\begin{remark}
    Theorem \ref{main-sup} implies that each supercuspidal representation associated to an unramified cuspidal datum can be realized by cohomology of higher Deligne-Lusztig varieties. The theorem is first proved in \cite{Nie_24} when $q$ is sufficiently large.
\end{remark}

In \cite{Kaletha_19}, Kaletha studied an important large family of supercupidal representations, called the regular supercupidal representations. These representations $\pi^{\rm Yu}_{(T, \phi)}$ are parameterized by tame elliptic regular pairs $(T, \phi)$, where $T \subseteq G$ is a tame elliptic maximal torus and $\phi$ is of character of $T(k)$ satisfying certain regularity conditions. We refer to \cite[\S 3]{Kaletha_19} for the precise definitions. Let $(\L, \phi_{-1})$ be a Howe factorization of $\phi$. Following \cite{ChanOi_25a} we set $\e[\phi] = \widetilde\e_\L |_{T(k)}$. The following result sharpens Theorem \ref{main-sup} for regular supercupidal representations.
\begin{corollary} \label{main-reg}
    Assume $p$ satisfies (*) and $q \ge c_\L$. Let $(T, \phi)$ be a tame elliptic regular pair such that $T$ is $\brk$-split. Then \[\pi^{\rm Yu}_{(T, \phi \cdot \e[\phi])} \cong  (-1)^{r(G^0) - r(T) + d_\L} \text{c-}\ind_{Z_G(k)\CG(\CO_k)}^{G(k)} R_{T, r}^G(\phi)\] where $r(G^0)$ and $r(T)$ are the split ranks of $G^0$ and $T$ respectively, and $\CG = \CG_\bx$ for some/any point in the apartment of $T$.
\end{corollary}

\begin{remark}
    Theorem \ref{main-sup} and Corollary \ref{main-reg} are also obtained by Chan and Oi \cite{ChanOi_25b} (see \cite{ChanOi_25a} when $\phi$ is 0-toral) under a largeness condition on $q$, see Remark \ref{compare}.
\end{remark}

\subsection{Strategy}
We outline the proof of Theorem \ref{main}. First we show that the two representations $(-1)^{d_\L}\k_\L$ and $\k(\L) \otimes \e_\L$ differs by a character $\psi$ of $(\BG^0)_0^F$, where $(\BG^0)_0$ is the reductive quotient of $\CG^0$. This is achieved using the key observation that both of them are extensions of a common Heisenberg representation. It remains to show the character $\psi$ is trivial. By our assumption on $q$, the group $(\BG^0)_0^F$ is generated by its commutator subgroup and $\BS_0^F$, where $\BS_0$ is a maximally split maximal torus of $\BG_0^0$. Thus, it remains to show $\psi$ is trivial over $\BS_0^F$. To this end, its suffices to that \[\tr(\g; \k(\L) \otimes \e_\L) = (-1)^{d_\L}\tr(\g; \k_\L) \neq 0 \text{ for all } \g \in \BS_0^F.\] The left trace is computed in \cite{AdlerS_09} and \cite{DS18}. The right trace is computed in Proposition \ref{geo-tr} based on a concentration theorem of \cite{Nie_24}. It turns out that both traces coincide with each other and are nonzero, which finishes the proof of the theorem.

\subsection*{Acknowledgment}
We would like to thank Alexander Ivanov for helpful discussions and comments.

\section{Preliminaries}
\subsection{Loop functor} Let $k$ be a non-archimedean local field with a finite residue field $\BF_q$ of cardinality $q$ and of characteristic $p > 2$. We denote by $\brk$ the completion of a maximal unramified extension of $k$. Let $F$ be the Frobenius automorphism of $\brk$ over $k$. Let $\CO_k$ and $\COk$ be the integer rings of $k$ and $\brk$ respectively.

Let $\CX$ be an affine $\CO_\brk$-scheme of finite type. Applying the perfect positive loop functor $L^+$ to $\CX$ gives a perfect affine $\ov\BF_q$-scheme \[\BX = L^+\CX\] such that for any perfect $\BF_q$-algebra $R$ we have $\BX(R) = \CX(\BW(R))$. Here $\BW(R)$ is the Witt ring of $R$ if ${\rm char}~k = 0$ and $\BW(R) = R[[\varpi]]$ otherwise. If $\CX$ is defined over $\CO_k$, then $\BX$ is an $\BF_q$-scheme and the Frobenius action on $\BX$ is still denoted by $F$.

\subsection{Parahoric subgroup} \label{subsec:parahoric} Let $G$ be a connected $k$-rational reductive group. By abuse of notation, we identify $G = G(\brk)$ and hence $G(k) = G^F$. Let $\bx$ be a point in the (enlarged) Bruhat-Tits building $\CB(G, k)$ of $G$ over $k$. Let $\widetilde\BR = \BR \sqcup \{s+; s \in \BR\}$ with the usual linear order $s < s+ < r$ for any $s < r \in \BR$. Thanks to the Bruhat-Tits building theory, there is an associated connected parahoric $\CO_k$-model $\CG = \CG_\bx$ of $G$, together with a filtration of Moy-Prasad subgroups $\CG^s$ for $s \in \widetilde\BR_{\ge 0}$. We set $\BG = L^+ \CG$ and $\BG^s = L^+ \CG^s$. For $r \in \BR_{\ge s}$, we define \[\BG_r^s = \BG^s / \BG^{r+},\] which are affine group schemes of finite type over $\BF_q$.

Moreover generally, for any $\brk$-rational closed subgroup $H \subseteq G$ we can consider its closure $\CH$ in $\CG$. Applying the functor $L^+$ to $\CH$ and then passing to subquotients yield closed subgroups $\BH_r^s$ of $\BG_r^s$ defined over $\ov\BF_q$. We put $\BH = \BH^s$ and $\BH_r = \BH_r^s$ when $s = 0$.

Let $S \subseteq G$ be a $k$-rational $\brk$-split maximal torus. We denote by $\CA(S, k)$ the apartment of
$S$ over $k$. Let $\Phi(S, G)$ and $\tPhi(G, S)$ be the sets of (absolute) roots and affine roots of $S$ in $G$, respectively. Each affine root $f$ in $\tPhi(G, S)$ defines an affine function on $\CA(S, k)$, which is still denoted by $f$. We denote by $\a_f \in \Phi(S, G)$ the gradient of $f$.

\section{Construction of $\k(\L)$} \label{sec:WH-rep}
A generic datum is a triple $\L = (G^i, \phi_i, r_i)_{0 \le i \le d}$, where \begin{itemize}
    \item $G^0 \subsetneq G^1 \subsetneq \cdots \subsetneq G^d = G$ are unramified Levi subgroups;

    \item $0 = r_{-1} < r_0 < \cdots < r_{d-1} \le r_d$ if $d \ge 1$ and $0 \le r_0$ if $d = 0$;

    \item $\phi_i: G^F \to \ov\BQ_\ell^\times$ is a character which is trivial over $G_{\rm sc}^F$ for $0 \le i \le d$;

    \item $\phi_i$ is of depth $r_i$ and is $(G^i, G^{i+1})$-generic in the sense of \cite[\S 9]{Yu_01} for $0 \le i \le d-1$;

    \item $\phi_d$ is of depth $r_d$ if $r_{d-1} < r_d$ and is trivial otherwise.
\end{itemize}
Moreover, we say $\L$ is unramified if $G^0$ splits over $\brk$.

\

In this section, we fix a generic datum $\L = (G^i, \phi_i, r_i)_{0 \le i \le d}$ as above and assume it is unramified. We set $L_\L = G^0$ and let $\bx \in \CB(L_\L, k)$. For closed subgroups $H \subseteq G$ we denote by $\BH$ the group attached to $\bx$ as in \S \ref{subsec:parahoric}.

Let $\CS_{\L, \bx}$ be the set of $k$-rational and $\brk$-split maximal tori $S \subseteq L_\L$  such that $\bx \in \CA(S, k)$.

\subsection{Subgroups of Yu-type} Following \cite{Yu_01} we introduce the following $F$-stable subgroups of $\BG$. \begin{align*}
\BK_\L &= \BG^0 (\BG^1)^{r_0/2} \cdots (\BG^d)^{r_{d-1}/2}; \\
\BH_\L &= (\BG^0)^{0+} (\BG^1)^{r_0/2} \cdots (\BG^d)^{r_{d-1}/2}; \\ \BK_\L^+ &= (\BG^0)^{0+} (\BG^1)^{r_0/2+} \cdots (\BG^d)^{r_{d-1}/2+}; \\ \BE_\L &= (\BG^0_\der)^{0+,0+} (\BG^1_\der)^{r_0+, r_0/2+} \cdots (\BG^d_\der)^{r_{d-1}+, r_{d-1}/2+}. \end{align*} Here   $(\BG^i_\der)^{r_{i-1}+, r_{i-1}/2+}$ is the subgroup generated by $(\BG^i_\der)^{r_{i-1}+}$ and $(\BG^\a)^{r_{i-1}/2+}$ for $\a \in \Phi(G^i, S) \setminus \Phi(G^{i-1}, S)$, where $S$ is any/some maximal torus in $\CS_{\L, \bx}$ and $G^\a \subseteq G$ is the root subgroup corresponding to $\a$.

For $0 \le i \le d$ let $\hat \phi_i$ denote the extended character of $\phi_i$ to the group $(\BG^i)^F (\BG^{r_i/2+})^F$ defined in \cite[\S 4]{Yu_01}.  We define a character of $(\BK_\L^+)^F$ by \[\chi_\L = \prod_{i=0}^d \hat \phi_i |_{(\BK_\L^+)^F}.\] Let $S \in \CS_{\L, \bx}$. Then $\BK_\L^+ = \BS^{0+} \BE_\L$ and $(\BK_\L^+)^F = (\BS^{0+})^F \BE_\L^F$. Hence $\chi_\L$ is also the inflation of the character $\prod_{i=0}^d \phi_i|_{(\BS^{0+})^F}$ (which is trivial over $(\BS^{0+} \cap \BE_\L)^F)$ under the quotient map $(\BK_\L^+)^F \to (\BK_\L^+)^F / \BE_\L^F \cong (\BS^{0+})^F / (\BS^{0+} \cap \BE_\L)^F$.


\subsection{Heisenberg representation}
We recall the construction of the Heisenberg representation of $\BH_\L^F$.

First we need the following result proved in \cite{Kim}.
\begin{proposition} \label{H-group}
      If $d = 0$ (and hence $\phi_d = 1$), the quotient $\BH_\L^F / \ker \chi_\L$ is trivial. Otherwise, it is a Heisenberg $p$-group whose center is $(\BK_\L^F) / \ker \chi_\L$.
\end{proposition}

Now we define an irreducible representation $\o_\L$ of $\BH_\L^F / \ker \chi_\L$ as follows. If $\BH_\L^F / \ker \chi_\L$ is trivial, let $\o_\L$ be the trivial representation. Otherwise, by Proposition \ref{H-group}, $\BH_\L^F / \ker \chi_\L$ is a Heisenberg $p$-group with center $(\BK_\L^+)^F / \ker \chi_\L$, and let $\o_\L$ be the unique irreducible (Heisenberg) representation of  $\BH_\L^F / \ker \chi_\L$ with the nontrivial central character $\chi_\L|_{(\BK_\L^+)^F / \ker \chi_\L}$. By inflation, we also view $\o_\L$ as a representation of $\BH_\L^F$.

\subsection{A trace formula on $\k(\L)$}  \label{subsec:FKS}
Let $[x]$ be the image of $\bx$ in $\CB((L_\L)_\der, k)$. Denote by $\widetilde \BL_\L$ the stabilizer of $[x]$ in $G$. Let $\widetilde \BK_\L = \widetilde \BL_\L \BH_\L = \BH_\L \widetilde \BL_\L$. Following \cite{Yu_01}, one can extend the $\BH_\L^F$-module $\o_\L$ to a $\widetilde\BK^F$-module $\tilde\k(\L)$ on which $(\widetilde\BL_\L)^F$ acts via the character $\prod_{i=0}^d \phi_i |_{(\widetilde \BL_\L)^F}$ times the (induced) Weil representation attached to $\o_\L$, see \cite[\S 2.5]{Fintzen_21} for the precise construction. We set $\k(\L) = \tilde \k(\L) |_{\BK_\L^F}$.

Following \cite{FintzenKS}, one can associated to $\L$ a quadratic character \[\e_\L: (\widetilde \BL_\L)_0^F \to \{\pm 1\}.\] By inflation, we also view $\e_\L$ as a character of $(\widetilde \BK_\L)^F$.

Let $S \in \CS_{\L, \bx}$ and $\g \in \BS$. We denote by $R_{S, i}(\g)$ the set of affine roots $f \in \tPhi(G^{i+1}, S) \sm \tPhi(G^i, S)$ such that $f(\bx) = r_i/2$ and $\a_f(\g)-1 \in \CP_\brk$. Here $\CP_\brk \subseteq \CO_\brk$ denotes the maximal ideal. We define $r(S, \L, \g) = \sum_{i = 0}^{d-1} |R_{S, i}(\g) / \<F\>|$. We set $R_{S, i} = R_{S, i}(\g)$ and $r(S, \L) = r(S, \L, \g)$ if $\g = 1$.

\begin{proposition} \label{WH-tr}
Let $S \in \CS_{\L, \bx}$ and $\g \in \BS^F$ be an unramified very regular element. Then \[\tr(\g; \k(\L)) = (-1)^{r(S, \L) - r(S, \L, \g)} |((\BH_\L/\BK_\L^+)^\g)^F|^{\frac{1}{2}} \e_\L(\g) \prod_{i=0}^d \phi_i(\g) \neq 0,\] where $(\BH_\L/\BK_\L^+)^\g$ denotes the set of points in $\BH_\L/\BK_\L^+$ which are fixed by the conjugation action of $\g$.
\end{proposition}
\begin{proof}
By \cite[Proposition 3.8]{AdlerS_09} and the formula ($\ddagger_0$) in the proof of \cite[Proposition 4.3.8]{DS18}, we have \begin{align*} &\quad\ \tr(\g; \k(\L)) \prod_{i=0}^d \phi_i(\g)\i = \prod_{i=0}^{d-1} \th_{\tilde \phi_i}(\g \ltimes 1)  \\ &=\prod_{i=0}^{d-1} |(C_{G^i}^{(0+)}(\g), C_{G^{i+1}}^{(0+)}(\g))_{\bx, (r_i, r_i/2) : (r_i, r_i/2+)}|^{\frac{1}{2}} \varepsilon(\phi_i, \g) \\ &= \prod_{i=0}^{d-1} |(C_{G^i}^{(0+)}(\g), C_{G^{i+1}}^{(0+)}(\g))_{\bx, (r_i, r_i/2) : (r_i, r_i/2+)}|^{\frac{1}{2}} (-1)^{|\dot \Xi_{\rm symm}(\phi_i, \g)|} \e_{\natural, \bx }^{G^{i+1}/G^i}(\g),\end{align*} where $\e_{\natural, \bx }^{G^{i+1}/G^i}$ is as in \cite[Definition 3.1]{FintzenKS} and all the other notations are defined as in \cite{DS18} by taking $(T, \phi, G', G) = (S, \phi_i, G^i, G^{i+1})$. As $S$ is unramified, $\dot \Xi_{\rm symm}(\phi_i, \g) \equiv (R_{i, S} \sm R_{i, S}(\g)) / \<F\> \mod 2$ for $0 \le i \le d-1$ and $\e_\L(\g) = \prod_{i=0}^{d-1}  \e_{\natural, \bx }^{G^{i+1}/G^i}(\g)$. Moreover, we have \[((\BH_\L/\BK_\L^+)^\g)^F \cong \prod_{i=0}^{d-1} (C_{G^i}^{(0+)}(\g), C_{G^{i+1}}^{(0+)}(\g))_{\bx, (r_i, r_i/2) : (r_i, r_i/2+)}.\] Hence the statement follows.
\end{proof}

\section{Construction of $\k_\L$}
Let notation be as in \S \ref{subsec:kappa}. Let $B$ and $\ov B$ be two $\brk$-rational opposite Borel subgroups of $G$ whose unipotent radicals are denoted by $U$ and $\ov U$ respectively. We assume moreover that $L_\L$ is a sytandard Levi subgroup with respect to $B$.

\subsection{The variety $Y_\L$} Recall that $\L = (G^i, \phi_i, r_i)_{0 \le i \le d}$. Let $\bar \BK_\L = \bar \BK_{\L, r} = \BK_\L / (\BE_\L \BG^{r_d+})$. For any subset $D \subseteq \BK_\L$ we denote by $\bar D$ its image under the quotient map $\BK_\L \to \bar \BK_\L$.

Let $\BI_{\L, U} = (\BK_\L \cap \BU) \BE_\L (\BK_\L^+ \cap \ov \BU) \subseteq \BK_\L$ be a subgroup. By the choice of $B$, $\BI_{\L, U}$ is normalized by $\BL_\L$. Consider the following variety \[Y_\L = Y_{\L, U} = \{y \in \bar \BH_\L; y\i F(y) \in F \bar\BI_{\L, U}\}.\]  As $[\BK_\L, \BK_\L^+] \subseteq \BE_\L \subseteq \BK_\L^+ \subseteq \BI_{\L, U}$ and that $\BL_\L$ normalizes $\BI_{\L, U}$, the group $(\BL_\L^F \ltimes \BH_\L^F) \times (\BK_\L^+)^F$ acts on $Y_\L$ by $(l, h, x): z \mapsto l h z l\i x$. Thus for each $i$ the $\chi_\L$-isotypic part $H_c^i(Y_\L, \ov\BQ_\ell)[\chi_\L]$ is a smooth representation of $\BL_\L^F \ltimes \BH_\L^F$. Let $\th_\L: \BL_\L^F \ltimes \BH_\L^F \to \ov\BQ_\ell^\times$ be the character given by $\th_\L(l \ltimes h) = \prod_{i=0}^d \phi_i(l)$. Then the representations of $\BL_\L^F \ltimes \BH_\L^F$ \[ H_c^i(Y_\L, \ov\BQ_\ell)[\chi_\L] \otimes \th_\L\] descend to smooth representations of $\BK_\L^F = \BL_\L^F \BH_\L^F$. We define \[\k_\L = \sum_i (-1)^i H_c^i(Y_\L, \ov\BQ_\ell)[\chi_\L] \otimes \th_\L\] as a virtual representation of $\BK_\L^F$.

\begin{theorem} \label{geo-rep}
    Let $S \in \CS_{\L, \bx}$ and $\g \in \BS^F$. There exists a unique integer $N = N_{\L, U, \bx, \g}$ such that \[H_c^i(Y_\L^\g, \ov\BQ_\ell)[\chi_\L] \neq 0 \text{ if and only if } i = N.\] In this case, $\dim H_c^i(Y_\L^\g, \ov\BQ_\ell)[\chi_\L] = |((\BH_\L / \BK_\L^+)^\g)^F|^{\frac{1}{2}}$ and $N \equiv r(S, \L, \g) \mod 2$.

    If, moreover, $\g = 1$, then $(-1)^N \k_\L \cong \o_\L$ and we can define $d_\L \in \{\pm 1\}$ such that $d_\L \equiv r(S, \L) \mod 2$ for some/any $S \in \CS_{\L, \bx}$.
\end{theorem}
\begin{proof}
    When $\g = 1$, the existence of $N$ and the isomorphism $(-1)^N \k_\L \cong \o_\L$ is proved by \cite[Theorem 6.2]{Nie_24} and \cite[Proposition 7.3]{Nie_24} respectively. For arbitrary $\g \in \BS^F$, the existence of $N$ follows in the same way. We indicate how to deduce the parity formula of $N$ and the dimension formula of $H_c^i(Y_\L^\g, \ov\BQ_\ell)[\chi_\L]$.

    Let $\Phi(G, S)^+$ and $\Phi(G, S)^-$ be the set of roots appearing in $U$ and $\ov U$ respectively. Let $\CC_i'$ (resp. $\CC_i''$) be the set of $F$-orbits $\CO$ in $R_{i, S}(\g)$ such that $\CO = r_i - \CO$ (resp. $\CO \neq r_i - \CO$. By arguments in \cite[\S 6.3]{Nie_24} and the proof of \cite[Theorem 6.2]{Nie_24}, it follows that \[\dim H_c^i(Y_\L^\g, \ov\BQ_\ell)[\chi_\L] = \prod_{i=0}^{d-1} \prod_{\CO \in R_{i, S}(\g) / \<F\>} q^{|\CO|/2} = |((\BH_\L / \BK_\L^+)^\g)^F|^{\frac{1}{2}}.\] Moreover, the existence of $N$ follows in the same way, and we have \[N \equiv \sum_{i=0}^{d-1} \sum_{\CO \in \CC_i'} |\CO^- \cap F \CO^+| \mod 2,\] where $\CO^\pm$ consists of affine roots $f \in \CO$ whose gradient $\a_f$ lies in $\Phi(G, S)^\pm$. Note that $|\CC_i''|$ is even and $|\CO^- \cap F \CO^+|$ is odd for $\CO \in \CC_i'$. Thus \[N \equiv \sum_{i=0}^{d-1} |\CC_i'| \equiv \sum_{i=0}^{d-1} |\CC_i'| + |\CC_i''| = r(S, \L, \g) \mod 2.\] The proof is finished.
\end{proof}

Let $S \in \CS_{\L, \bx}$. Then $\BK_\L^+ = \BS^{0+} \BE_\L$ and $(\BK_\L^+)^F = (\BS^{0+})^F \BE_\L^F$. We denote by $H_c^i(Y_\L, \ov\BQ_\ell)[\chi_\L|_{(\BS^{0+})^F}]$ the $\chi_\L|_{(\BS^{0+})^F}$-isotypic part of $H_c^i(Y_\L, \ov\BQ_\ell)$.
\begin{lemma} \label{weight}
    Let $S \in \CS_{\L, \bx}$. Then $H_c^i(Y_\L, \ov\BQ_\ell)[\chi_\L|_{(\BS^{0+})^F}] = H_c^i(Y_\L, \ov\BQ_\ell)[\chi_\L]$ for $i \in \BZ$.
\end{lemma}
\begin{proof}
    It follows from that (1) $\BE_\L$ acts trivially on $Y_\L$ by right multiplication, (2) $\chi_\L$ is trivial over $\BE_\L^F$ and (3) $(\BK_\L^+)^F = (\BS^{0+})^F \BE_\L^F$.
\end{proof}

\subsection{A trace formula on $\k_\L$}
We have the following trace formula for $\k_\L$.
\begin{proposition} \label{geo-tr}
    Let $S \in \CS_{\L, \bx}$ and $\g \in \BS^F$. Then \[\tr(\g; \k_\L) = (-1)^{r(S, \L, \g)} |((\BH_\L / \BK_\L^+)^\g)^F|^{\frac{1}{2}} \prod_{i=0}^d\phi_i(\g).\]
\end{proposition}
\begin{proof}
    Let $\g_s$, $\g_u$ be the semisimple and unipotent part of $\g$ respectively. Then $\g_s \in \BS^F$ and $\g_u \in (\BS^{0+})^F$. Hence $\bar \BH_\L^{\g_s} = \bar \BS^{0+}$ and $Y_\L^{\g_s} = (\bar \BS^{0+})^F$. By construction of $\k_\L$ in \S\ref{subsec:kappa}, we deduce that \begin{align*} &\quad\  \tr(\g; \k_\L) \prod_{i=0}^d\phi_i(\g\i) \\ &= \tr(\g; H_c^*(Y_\L, \ov\BQ_\ell)[\chi_\L]) \\ &=\tr(\g; H_c^*(Y_\L, \ov\BQ_\ell)[\chi_\L |_{(\BS^{0+})^F}])  \\ &= \frac{1}{|(\BS^{0+})^F|} \sum_{t \in (\BS^{0+})^F} \chi_\L(t\i) \tr\big( (\g, t); H_c^*(Y_\L, \ov\BQ_\ell)\big) \\ &= \frac{1}{|(\BS^{0+})^F|} \sum_{t \in (\BS^{0+})^F} \chi_\L(t\i) \tr\big( (\g_u, t); H_c^*(Y_\L^{\g_s}, \ov\BQ_\ell)\big) \\  &= \frac{1}{|(\BS^{0+})^F|} \sum_{t \in (\BS^{0+})^F} \chi_\L(t\i) \tr\big((1, t); H_c^*(Y_\L^{\g_s}, \ov\BQ_\ell) \big) \\  &= \dim H_c^*(Y_\L^{\g_s}, \ov\BQ_\ell)[\chi_\L] \\ &= (-1)^{r(S, \L, \g)} |((\BH_\L / \BK_\L^+)^\g)^F|^{\frac{1}{2}}, \end{align*} where the second equality follows from Lemma \ref{weight}; the fifth one follows from that $\g_u \in (\BS^{0+})^F$ acts on $Y_\L$ trivially by conjugation; the last one follows from Theorem \ref{geo-rep}.
\end{proof}

\section{Comparision between $\k_\L$ and $\k(\L)$}
In this section we complete the proof of Theorem \ref{main}.

\subsection{Results on inductions}
First we recall two general results on representations of finite groups.
\begin{lemma} \label{irr-dec}
    Let $K$ be a finite group and $H \subseteq K$ a normal subgroup. Let $\k$ be a $K$-module such that $\k|_H$ is an irreducible $H$-module. Then \[\ind_H^K \k|_H \cong \k \otimes \ind_H^K 1 \cong \bigoplus_\rho \k \otimes \rho,\] where $\rho$ ranges over irreducible summands (with multiplicities) of $\ind_H^K 1$. Moreover, all the sub-modules $\k \otimes \rho$ are irreducible, and $\k \otimes \rho \cong \k \otimes \rho'$ if and only if $\rho \cong \rho'$.
\end{lemma}

\begin{proposition} \label{diff}
    Let $H \subseteq K$ be as in Lemma \ref{irr-dec}. Let $\k, \k'$ be two $K$-modules such that $\dim \k = \dim \k'$ and $\k|_H$, $\k'|_H$ are isomorphic irreducible $H$-modules. Then $\k' \cong \k \otimes \psi$ for some character $\psi$ of $H/K$.
\end{proposition}
\begin{proof}
    As $\k|_H \cong \k'|_H$, we have \[\hom_K(\ind_H^K \k|_H, \k') \cong \hom_H(\k|_H, \k'|_H) \neq 0.\] Applying Lemma \ref{irr-dec} we deduce that $\k' \cong \k \otimes \rho$ for some irreducible summand $\rho$ of $\ind_H^K 1$. As $\dim \k = \dim \k'$ it follows that $\dim \psi = 1$. Hence $\psi$ is a character of $H$ which factors through the quotient $K / H$.
\end{proof}

\subsection{Proof of the main result}
Let notation be as in \S \ref{sec:WH-rep}.
\begin{proposition} \label{der}
    There exists an integer $2 \le c_\L \le 4$ such that when $q \ge c_\L$ the group $(\BL_\L)_0^F$ is generated by $\BS_0^F$ and the commutator subgroup $[(\BL_\L)_0^F, (\BL_\L)_0^F]$, where $S \in \CS_{\L, \bx}$ such that $\BS_0$ is a maximally split maximal torus of $(\BL_\L)_0$.
\end{proposition}
\begin{proof}
    Let $\BL_0 = (\BL_\L)_0$ By assumption, there exist an $F$-stable opposite maximal unipotent subgroups $\BV_0, \ov\BV_0 \subseteq \BL_0$ which are normalized by $\BS_0$. Then $\BL_0^F$ is generated by $\BS_0^F$, $\BV_0^F$ and $\ov\BV_0^F$. By the proof of \cite[Lemma 8.6]{Nie_24}, we have $\BV_0^F \subseteq [\BS_0^F, \BV_0^F]$ and $\ov\BV_0^F \subseteq [\BS_0^F, \ov\BV_0^F]$ if for each root $\a \in \Phi(\BL_0, \BS_0)$ there exists $z \in \BS_0^F$ such that $\a(z) \neq 1$. By a direct case-by-case analysis on the irreducible Dynkin diagrams, we deduce that this condition is always satisfied when $q \ge 4$. Hence the statement follows.
\end{proof}


\begin{theorem} \label{main-kappa}
    There exists a character $\psi$ of $\BK_\L^F$, which factors through $(\BL_\L)_0^F = \BK_\L^F / \BH_\L^F$, such that \[(-1)^{d_\L} \k_\L \cong \k(\L) \otimes \e_\L \otimes \psi\] as genuine representations of $\BK_\L^F$. If, moreover, $q \ge c_\L$, then $\psi$ is trivial. Here $c_\L$ is as in Proposition \ref{der}.
\end{theorem}
\begin{proof}
    By Proposition \ref{geo-rep}, $(-1)^{r(S, \L)} \k_\L |_{\BH_\L^F}$ is a genuine irreducible representation isomorphic to $\o_\L$. On the other hand, $(\k(\L) \otimes \e_\L) |_{\BH_\L^F} \cong \k(\L) |_{\BH_\L^F}$ is also isomorphic to $\o_\L$. Applying Proposition \ref{diff}, there exists a character $\psi$ of $\BK_\L^F / \BH_\L^F \cong (\BL_\L)_0^F$ such that \[\tag{a} \k(\L) \otimes \e_\L \cong (-1)^{d_\L} \k_\L \otimes \psi.\]

    Let $S \in \CS_{\L, \bx}$ such that $\BS_0$ is a maximally split maximal torus of $(\BL_\L)_0$. Let $\g \in \BS^F$. By Proposition \ref{geo-tr} and Proposition \ref{WH-tr}, we have $(-1)^{r(S, \L)} = (-1)^{d_\L}$ and \[\tr(\g; \k(\L) \otimes \e_\L) = \tr(\g; (-1)^{r(S, \L)} \k_\L) \neq 0.\] By (a), $\psi(\g) = 1$ and hence $\psi$ is trivial over $\BS^F$. By Proposition \ref{der}, $\psi$ is trivial and the statement follows.
\end{proof}

\section{Applications}
We discuss applications of Theorem \ref{main} on supercuspidal representations of $p$-adic group. We assume  the condition (*) throughout this section.

\subsection{Higher Deligne-Lusztig representation} \label{subsec:hDL}
let $T \subseteq G$ be an unramified elliptic maximal torus such that $\bx \in \CA(T, k)$. Let $\phi: T^F \to \ov\BQ_\ell^\times$ be a smooth character of depth $\le r$, where $r$ is a nonnegative integer.

Following \cite{CI_MPDL} the higher Deligne-Lusztig variety is defined by \[X_r = X_{T, U, r} =\{g \in \BG_r; g\i F(g) \in F \BU_r\},\] which admits a natural action of $\BG^F \times \BT^F$ given by $(g, t): x \mapsto g x t$. Then the group $\BG^F \times \BT^F$ acts on the $\ell$-adic cohomological groups with compact support $H_c^i(X_r, \ov\BQ_\ell)$ for $i \in \BZ$. Hence the $\phi |_{\BT^F}$-isotypic parts $H_c^i(X_r, \ov\BQ_\ell)[\phi |_{\BT^F}]$ are smooth representations of $\BG^F$.

The associated higher Deligne-Lusztig representation is defined by \[R_{T, U, r}^G(\phi) = \sum_i (-1)^i H_c^i(X, \ov\BQ_\ell)[\phi|_{\BT^F}],\] which is viewed as a virtual smooth representation of $Z_G^F\BG^F$. Thanks to \cite{Chan24}, $R_{T, U, r}^G(\phi)$  is independent of the choice of $r$ or $U$. Hence we denote it by $R_\BT^\BG(\phi)$.

\subsection{A decomposition result} \label{subsec:kappa}
Let $(\L, \phi^{-1})$ be a Howe factorization of $\phi$, namely \begin{itemize}
    \item $\phi_{-1}: G^F \to \ov\BQ_\ell^\times$ is a character of depth zero;

    \item $\L = (G^i, \phi_i, r_i)_{0 \le i \le d}$ is an unramified generic datum such that $T \subseteq G^0$;

    \item $\phi = \prod_{i= -1}^d \phi_i |_{T^F}$.
\end{itemize}
Thanks to \cite{Kaletha_19}, Howe factorizations always exist under the condition (*) on $p$.

Note that $\BT_0$ is an $\BF_q$-rational maximal torus of the reductive quotient $(\BL_\L)_0$ of $\BL_\L$, and $\phi_{-1}$ descends to character of $\BT_0^F$. We denote by $R_\BT^{\BL_\L}(\phi_{-1})$ the virtual $\BK_\L^F$-module inflated from the classical Deligne-Lusztig representation of $(\BL_\L)_0^F$ attached to the pair $(\BT_0^F, \phi_{-1})$ as in \cite{DeligneL_76}.

We also view the $\BK_\L^F$-modules $\k_\L$ and $R_\BT^{\BL_\L}(\phi_{-1})$ as $Z_G^F \BK_\L^F$-modules on which $Z_G^F$ acts by the characters $\prod_{i=0}^d \phi_i |_{Z_G^F}$ and $\phi_{-1}|_{Z_G^F}$ respectively.
\begin{theorem} [{\cite[Theorem 1.6]{Nie_24}}] \label{decomp}
    Let $(\L, \phi_{-1})$ be a Howe factorization of $\phi$. Then \[R_\BT^\BG(\phi) = \ind_{ \BK_\L^F}^{\BG^F} \k_\L \otimes R_\BT^{\BL_\L}(\phi_{-1}) = \sum_\rho m_\rho \ind_{ \BK_\L^F}^{\BG^F} \k_\L \otimes \rho,\] where $\rho$ ranges over irreducible $(\BL_\L)_0^F$-modules up to equivalence, and $m_\rho$ is the coefficient of $\rho$ in $R_\BT^{\BL_\L}(\phi_{-1})$.

    Moreover, the representations $\ind_{\BK_\L^F}^{\BG^F} \k_\L \otimes \rho$ are pairwise non-isomorphic irreducible $\BG^F$-modules.
\end{theorem}

\subsection{Yu's construction}
Following \cite{Yu_01}, we say a triple $\Sigma = (\L, \bx, \widetilde\rho)$ is a cuspidal datum, namely \begin{itemize}
    \item $\L = (G^i, \phi_i, r_i)_{0 \le i \le d}$ is a generic datum with $Z_G / Z_{G^0}$ anisotropic;

    \item $\bx \in \CB(G^0, k)$ whose image $[\bx]$ in $\CB((G^0)_\der, k)$ is a vertex;

    \item $\widetilde\rho$ is a depth $0$ irreducible representation of $\widetilde\BL_\L^F$ such that $\widetilde\rho |_{(\BL_\L)^F}$ descends to a cuspidal representation of $(\BL_\L)_0^F$. Here $\widetilde \BL_\L$ is the stabilizer of $[x]$ in $L_\L = G^0$.
\end{itemize}
Moreover, we say $\Sigma$ is an unramified cuspidal datum if $\L$ is an unramified generic datum.

\begin{theorem} [{\cite[Theorem 0.1]{Yu_01}}]
    Let $\Sigma = (\L, \bx, \widetilde\rho)$ be a cuspidal datum. Set $\widetilde \BK_\L = \widetilde \BL_\L \BK_\L$. Then the compactly induced representation \[\pi_\Sigma^{\rm Yu} := \text{c-}\ind_{\widetilde \BK_\L^F}^{G^F} \widetilde\k(\L) \otimes \widetilde\rho\] is an irreducible supercuspidal representation of $G^F$.
\end{theorem}

Let $\Sigma = (\L, \bx, \widetilde\rho)$ be an unramified cupidal datum. By \cite[Corollary 7.7 \& Proposition 8.2]{DeligneL_76}, there exist an elliptic torus $T \in \CS_{\L, \bx}$ and a depth 0 character $\phi_{-1}$ of $\BT^F$ such that $\widetilde\rho |_{\BL_\L^F}$ appears in $R_{\BT}^{\BL_\L}(\phi_{-1})$. Moreover, as $T$ is unramified, we have $T^F = Z_G^F \BT^F$. Hence we can extends $\phi_{-1}$ to a character of $T^F$, still denoted by $\phi_{-1}$, such that the central character of $\rho$ equals $\phi_{-1} |_{Z_G^F}$. Let $\phi = \chi_\L|_{\BT^F} \phi_{-1}$. By definition $(\L, \phi_{-1})$ is a Howe factorization of $\phi$.

We view the $\BG^F$-module $R_\BT^\BG(\phi)$ as a $Z_G^F \BG^F$-module on which $Z_G^F$ acts via the character $\phi|_{Z_G^F}$.
\begin{theorem}
    Assume $p$ satisfies (*) and $q \ge c_\L$. Let $\Sigma$ and $\phi$ be as above. Then  \begin{itemize}
        \item $\pi_\Sigma^{\rm FKS}$ is a summand of $\text{c-}\ind_{Z_G^F \BG^F}^{G^F} R_\BT^\BG(\phi)$;

        \item $\pi_\Sigma^{\rm FKS} \cong (-1)^{r(L_\L) -r(T) + d_\L} \text{c-}\ind_{Z_G^F \BG^F}^{G^F} R_\BT^\BG(\phi)$ if $\phi_{-1}$ is regular in the sense of \cite[Definition 3.4.16]{Kaletha_19}.
    \end{itemize}
    Here $r(L_\L)$ and $r(T)$ denotes the $k$-rank of $L_\L$ and $T$ respectively.
\end{theorem}
\begin{proof}
    Let $A$ be finite group and let $\chi = \sum_\pi c_\pi \pi$ be a virtual $A$-module, where $c_\pi \in \BZ$ and $\pi$ ranges over irreducible $A$-modules up to isomorphism. We set $|\chi| = \sum_\pi |c_\pi| \pi$.

    By reciprocity and the choice of $\phi_{-1}$ we have \[\hom_{\widetilde \BL_\L^F}(\widetilde \rho, \ind_{Z_G^F \BL_\L^F}^{\widetilde \BL_\L^F} |R_\BT^{\BL_\L}(\phi_{-1})|) = \hom_{Z_G^F \BL_\L^F} (\rho, |R_\BT^{\BL_\L}(\phi_{-1})|) \neq 0.\] Now we have \begin{align*}
        \pi_\Sigma^{\rm FKS} &= \text{c-}\ind_{\widetilde \BK_\L^F}^{G^F} \widetilde\k(\L) \otimes \widetilde \e_\L \otimes \widetilde \rho \\ &\subseteq \text{c-}\ind_{Z_G^F \BK_\L^F}^{G^F} \k(\L) \otimes \e_\L \otimes \otimes |R_\BT^{\BL_\L}(\phi_{-1})| \\ &=\text{c-}\ind_{Z_G^F \BK_\L^F}^{G^F} |\k_\L| \otimes |R_\BT^{\BL_\L}(\phi_{-1})| \\ &= \text{c-}\ind_{Z_G^F \BG^F}^{G^F} \ind_{Z_G^F \BK^F}^{Z_G^F \BG^F} |\k_\L| \otimes |R_\BT^{\BL_\L}(\phi_{-1})| \\
        &= \text{c-}\ind_{Z_G^F \BG^F}^{G^F} |R_\BT^\BG(\phi)|,
    \end{align*} where $\e_\L, \k(\L)$ are viewed as restrictions of $\widetilde \e_\L, \widetilde \k(\L)$ to $Z_G^F \BK_\L^F$ respectively, and the last equality follows from Theorem \ref{decomp}. Hence the first statement follows.

    Now assume that $\phi_{-1}$ is regular in the sense of \cite[Definition 3.4.16]{Kaletha_19}. As $Z_G^F \BL_\L^F = T^F \BL_\L^F$, it follows by \cite[Lemma 3.4.20]{Kaletha_19} that \[\widetilde \rho \cong (-1)^{r(L_\L) - r(T)}\ind_{Z_G^F \BL_\L^F}^{\widetilde \BL_\L^F} R_\BT^{\BL_\L}(\phi_{-1}).\] Note that $\widetilde \BL_\L^F / Z_G^F \BL_\L^F \cong \widetilde \BK_\L^F / Z_G^F \BK_\L^F$. Applying projection formula we have \begin{align*}\widetilde \k(\L) \otimes \widetilde \e_\L \otimes \widetilde \rho &\cong (-1)^{r(L_\L) - r(T)}\ind_{Z_G^F \BK_\L^F}^{\widetilde \BK_\L^F} \k(\L) \otimes \e_\L \otimes R_\BT^{\BL_\L}(\phi_{-1}) \\ &\cong (-1)^{r(L_\L) - r(T) + d_\L} \ind_{Z_G^F \BK_\L^F}^{\widetilde \BK_\L^F} \k_\L \otimes R_\BT^{\BL_\L}(\phi_{-1}). \end{align*} Then the second statement follows similarly as in the proof of the first statement.
\end{proof}

\end{document}